\documentclass[12pt,a4paper]{amsart}
\usepackage{amsfonts}
\usepackage{amsthm}
\usepackage{amsmath}
\usepackage{amscd}
\usepackage[latin2]{inputenc}
\usepackage{t1enc}
\usepackage[mathscr]{eucal}
\usepackage{indentfirst}
\usepackage{graphicx}
\usepackage{graphics}
\usepackage{pict2e}
\usepackage{epic}
\numberwithin{equation}{section}
\usepackage[left=1in,right=1in,top=1in,bottom=1in]{geometry}
\usepackage{epstopdf} 

\usepackage{hyperref}
\hypersetup{
	colorlinks = true,
    linkcolor = blue,
    citecolor = red,
    urlcolor = black,
}

\theoremstyle{plain}
\newtheorem{Th}{Theorem}[section]
\newtheorem{Lemma}[Th]{Lemma}

 \theoremstyle{definition}

\newtheorem{?}[Th]{Problem}

\newcommand{\R}{\mathbb R}

\newcommand{\C}{\mathbb C}

\begin{document}

\title{A better bound for ordinary triangles}
\author{Quentin Dubroff*}
\thanks{*Department of Mathematics, EPFL, Lausanne, Switzerland, email: \href{mailto:quentindubroff@gmail.com}{quentindubroff@gmail.com}. Partially supported by Swiss National Science Foundation grants 200020-162884 and 200021-165977}

\begin{abstract}
Let $P$ be a finite set of points in the plane. A \emph{c-ordinary triangle} is a set of three non-collinear points of $P$ such that each line spanned by the points contains at most $c$ points of $P$. We show that if $P$ is not contained in the union of two lines and $|P|$ is sufficiently large, then it contains an 11-ordinary triangle. This improves upon a result of Fulek et al., who showed one may take $c=12000$.
\end{abstract}

\maketitle

\vspace{-4mm}
\section{Introduction}

Let $P$ be a set of $n$ points in $\R^2$. A line which contains two or more points of $P$ is called a \emph{spanned line}, and we call a spanned line \emph{c-ordinary} if it contains at most $c$ points of $P$. The Sylvester-Gallai Theorem states that there exists a 2-ordinary line in any set of non-collinear points. Various proofs of this result have been found, and this work sparked an expansive literature, for which we refer the interested reader to \cite{GreenTao}.

Erd\H{o}s \cite{ErdosProblems} was the first to consider conditions under which one can find a \emph{2-ordinary triangle} in a set of points, that is to say three non-collinear points such that each spanned line is 2-ordinary. While it is clear that the point set must not be contained in the union of two lines, each containing at least three points, the following example illustrates that this is not sufficient \cite{OrdinaryTriangles}: Given a set of points $P$ and a line $l$ not spanned by these points, add to $P$ the points of intersection between $l$ and each of the lines spanned by $P$, to get a new set of points $P'$. Any 2-ordinary triangle in $P'$ must contain at least two points from the original point set $P$, however, each line spanned by $P$ contains a third point on $l$. Therefore, there are no 2-ordinary triangles in $P'$. In fact, a construction of F\"{u}redi and Pal\'{a}sti \cite{Arrangements} shows that one may not always find a 2-ordinary triangle, even under the restriction that every line is 3-ordinary.

With an interest in studying ordinary conics \cite{Conics}, de Zeeuw asked whether there exists an integer $c$ such that every set $P$ of points not contained in the union of two lines contains a \emph{c-ordinary triangle}, that is, three non-collinear points of $P$, where each line spanned by the points is $c$-ordinary. Fulek, Nassajian Mojarrad, Nasz\'{o}di, Solymosi, Stich, and Szedl\'{a}k \cite{OrdinaryTriangles} answered in the affirmative for $n$ sufficiently large, and showed one may take $c=12000$. We improve this to $c=11$.

\begin{Th}\label{mythm}
If $P$ is a finite set of $n$ points in $\R^2$ not contained in the union of two lines, with $n$ larger than 7697, then there exists an 11-ordinary triangle in $P$.
\end{Th}

It remains open whether one may find 3-ordinary triangles in any set of points not contained in the union of two lines. Improvement on Theorem \ref{mythm} would follow from improvements on any of the lemmas used in the paper, none of which are known to be the best possible. In addition, all of the tools used in the proof of Theorem \ref{mythm} have suitable analogs in $\C^2$, where the proof gives $c=12$. Here too it remains open whether one may take $c=3$. More generally, it may even be possible to find $k$-ordinary $k$-tuples in every point set not contained in the union of $k-1$ lines. This would be the best possible due to a straightforward generalization of the construction described above, obtained by adding to $P$ the points of intersection of the spanned lines of $P$ and $k-2$ other lines. 

\section{Tools}
In this section, we introduce a few results that are needed for the proof of Theorem \ref{mythm}. These results may be found in \cite{LinesLanger}, which collects several results in combinatorial geometry that follow from an inequality of Langer \cite{Langer} in algebraic geometry, stating that the number of incidences between a set of $n$ points and their spanned lines is at least $\frac{n(n+3)}{3}$, provided that at most $\frac{2n}{3}$ points lie on one line. Han \cite{Dirac} observed that this incidence inequality immediately implies the following improvement of the weak Dirac conjecture. We include the proof of this lemma, as it will play a role in the proof of Theorem \ref{mythm}.

\begin{Lemma}\label{weakDirac}

Let $P$ be a set of $n$ points in $\R^2$, not all on a line. Then there exists a point of $P$ that is contained in at least $\frac{n}{3} + 1$ lines spanned by $P$.

\end{Lemma}
\begin{proof}
If more than $\frac{2n}{3}$ points of $P$ lie on one line, then any point not on this line will be contained in more than $\frac{2n}{3}$ lines. If not, then the conditions of Langer's inequality hold, and the number of incidences between the points of $P$ and the set of lines spanned by $P$ is at least $\frac{n(n+3)}{3}$. By the pigeonhole-principle, some point of $P$ must be contained  in at least $\frac{n}{3}+1$ lines.
\end{proof}
Next, we introduce de Zeeuw's \cite{LinesLanger} improvement on Beck's \cite{Beck} theorem of two extremes.

\begin{Lemma}\label{Beck}

Given a set of $n$ points in $\R^2$, at least one of the following is true:

\begin{itemize}
\item \textit{There is some line containing more than $\gamma n$ points of $P$, where $\gamma = (6+\sqrt{3})/9$.}
\vspace{1mm}
\item \textit{There are at least $\frac{n^2}{9}$ lines spanned by $P$.}
\end{itemize}

\end{Lemma}

Lastly, de Zeeuw derived the following bound on the number of lines containing more than $k$ points of a point set. This lemma serves a similar purpose as the Szemer\'{e}di-Trotter Theorem, but gives better constants for small values of $k$ \cite{LinesLanger}.

\begin{Lemma}\label{linesbound}
Let $P$ be a set of $n$ points in $\R^2$ with no more than $\frac{2n}{3}$ points on a line, and let $\mathcal{L}$ be the set of lines spanned by $P$. Then the number of lines of $\mathcal{L}$ containing more than $k$ points of $P$ is at most
$$\frac{4}{(k-1)^2}
|\mathcal{L}|.$$
\end{Lemma}

\section{Proof of Theorem \ref{mythm}}
\noindent In this section we prove Theorem \ref{mythm}. Here $P$ is a set of points not contained in the union of two lines, $\mathcal{L} = \{L_1, \ldots , L_m\}$ is the set of spanned lines, $|L_i| = l_i$, and we let $\lambda = \frac{5/2}{c+1}$, for $c$ to be determined later. The proof is split into the following two cases:
\begin{enumerate}
\item \label{case1} There exists a line $L_i \in \mathcal{L}$ with $l_i \geq \lambda n$.

\item \label{case2} Every line contains less than $\lambda n$ points of $P$.
\end{enumerate}

The proof for the first case follows that of Fulek et al. \cite{OrdinaryTriangles}. Our main contribution is finding a better argument for the second case.

Consider case (\ref{case1}). The set of points $P\setminus L_i$ is non-collinear by assumption, so there must be a 2-ordinary line $\overline{qr}$ by the Sylvester-Gallai Theorem, where $\overline{qr}$ denotes the line determined by $q$ and $r$. Note that $\overline{qr}$ intersects $L_i$ in at most one point. We show that there are many points on $L_i$ which form $c$-ordinary triangles with $\overline{qr}$. We define the set $P_q \subset P$ as
$$P_q = \{ p \in L_i \cap P : |\overline{pq}\cap P| > c\},$$
\noindent We define $P_r$ in a similar way. Any point of $P_q$ determines a line of $\mathcal{L}$ containing at least $c-1$ points of $P\setminus (L_i\cup \{q\})$, and furthermore, these points are disjoint for each $p\in P_q$, so we have

$$(c-1)|P_q| \leq n - l_i,$$
implying
$$|P_q| \leq \frac{n-l_i}{c-1} \leq \frac{l_i/\lambda - l_i}{c-1} = \frac{2(c+1)/5 - 1}{c-1} l_i < \frac{2}{5}l_i.$$

Similarly, $|P_r| < \frac{2}{5}l_i$, so there are at least $\frac{l_i}{5}$ points $s\in P\cap L_i$ such that $s\notin P_q\cup P_r$ and $s,q,r$ are non-collinear. Each of these points together with $q$ and $r$ determines a $c$-ordinary triangle for $P$.

Now consider case (\ref{case2}). We will derive an upper bound on $c$ assuming that $P$ does not contain a $c$-ordinary triangle. For a point $q\in P$, we define the set $N_q \subset P$ as 
$$N_q = \{p \in P : |\overline{qp}\cap P| \leq c\}.$$
By Lemma \ref{weakDirac}, there exists a point $q$ of $P$ that lies on at least $n/3$ lines. Let the number of these lines be denoted by $l$. Any line incident to $q$ which is not $c$-ordinary contains at least $c$ points of $P\setminus \{q\}$, and, as in case (\ref{case1}), these points are disjoint for each line incident to $q$, so we find that
$$|N_q| + (l-|N_q|)c \leq n-1,$$
that is,
\begin{equation}\label{neighborhood}
|N_q| \geq \frac{cl - n + 1}{c-1}.
\end{equation}

From the conditions of case (\ref{case2}) and for $c$ at least 11, no line contains $\gamma|N_q|$ points of $N_q$, with $\gamma$ from Lemma \ref{Beck}, so by the lemma, $N_q$ must span at least $|N_q|^2 / 9$ lines. All of these lines which do not go through $q$ must contain at least $c+1$ points of $P$, or we will find a $c$-ordinary triangle. By Lemma \ref{linesbound}, the number of these lines is at most
$$\frac{4}{(c-1)^2}|\mathcal{L}|.$$
\noindent At most $l$ of the lines spanned by $N_q$ are incident to $q$, so 
$$\frac{|N_q|^2}{9} - l \leq \frac{4}{(c-1)^2}|\mathcal{L}|,$$
\noindent which implies that
$$|N_q| \leq \frac{6}{c-1}\sqrt{|\mathcal{L}|} + 3\sqrt{l}.$$
\noindent Substituting inequality \ref{neighborhood}, we find 
$$\frac{cl - n}{c-1} \leq \frac{6}{c-1}\sqrt{|\mathcal{L}|} + 3\sqrt{l},$$
\noindent and therefore
$$c \leq \frac{6\sqrt{|\mathcal{L}|} + n - 3\sqrt{l}}{l - 3\sqrt{l}} < \frac{6\sqrt{|\mathcal{L}|} + n}{l - 3\sqrt{l}}.$$

From Lemma \ref{weakDirac}, $l \geq \frac{n}{3}$, so when $|\mathcal{L}| \leq \frac{n^2}{6}$, we have
$$c < \frac{3(\sqrt{6}+1)}{1 - 3\sqrt{\frac{3}{n}}}.$$

\noindent If instead $|\mathcal{L}| > \frac{n^2}{6}$, the number of incidences between $P$ and $\mathcal{L}$ is at least $2|\mathcal{L}| > \frac{n^2}{3}$, so the argument from the second part of the proof of Lemma \ref{weakDirac} gives $l \geq \frac{2\mathcal{|L|}}{n}$. Letting $|\mathcal{L}| = \frac{n^2}{B^2}$, for $B \in [\sqrt{2}, \sqrt{6})$, we will have $c$-ordinary triangles unless
$$c \leq \frac{6n/B + n}{\frac{2n}{B^2} - 3\frac{\sqrt{2n}}{B}} = \frac{6B + B^2}{2-3B\sqrt{\frac{2}{n}}} < \frac{3(\sqrt{6}+1)}{1-3\sqrt{\frac{3}{n}}}$$

\noindent Therefore, for $n \geq 7697$, we will find $c$-ordinary triangles if $c \geq 11$.

\section{Acknowledgments}

I would like to thank Frank de Zeeuw for his helpful input and extensive comments on the manuscript, as well as Hossein Nassajian Mojarrad for suggesting the problem. In addition, I would like to thank J\'{a}nos Pach for his hospitality and support throughout my research stay.

\end{document}